\documentclass{article}
\usepackage[margin=1in]{geometry}
\usepackage{graphicx} 
\usepackage{amsmath,amssymb,amsthm}
\title{Focusing on gap lengths in the noisy violinist chip-firing problem}
\author{Leonid Ryvkin}
\date{\today}
\usepackage[style=numeric]{biblatex}
\newtheorem{lemma}{Lemma}
\newtheorem{thm}[lemma]{Theorem}
\newtheorem{cor}[lemma]{Corollary}
\newtheorem{defi}[lemma]{Definition}
\usepackage{forest}
\begin{document}

\maketitle
\begin{abstract}
    We encode the states in the noisy violinist chip-firing problem into a sequence of gap lengths, and use this gap-focused perspective to reprove certain statements about final states of flat clusterons.
\end{abstract}

\section{The noisy-violinist problem and its reformulation}
The following problem is suggested and solved in \cite{2302.11067}:

\begin{quote}
Consider a hotel with an infinite number of rooms arranged sequentially on the ground floor. The rooms are labeled from left to right by integers,
with room i being adjacent to rooms $i-1$ and $i+1$. Room $i-1$ is on the left
of room $i$, while room $i+1$ is on the right.\\
A finite number of violinists are staying in the hotel; each room has at most one
violinist in it. Each night, some two violinists staying in adjacent rooms (if two
such violinists exist) decide they cannot stand each other’s noise and move apart:
One of them moves to the nearest unoccupied room to the left, while the other
moves to the nearest unoccupied room to the right. This keeps happening for as
long as there are two violinists in adjacent rooms. Prove that this moving will
stop after a finite number of days.
\end{quote}

We suggest here a slightly different perspective to solving the same problem: We will describe the {\bf state} of the system (with $N$ violinists) by an $N-1$-tuple of numbers $a_\bullet=(a_1,...,a_{N-1})$. If we number the violinists by 1,...,k from left to right $a_k$ is the length of the gap between violinist $k$ and $k+1$.\\

This description does not contain the information of the absolute position of the leftmost violinist, (but we don't care about this information anyways). I.e. a state for us is what is referred to as the \emph{shadow of a state} in \cite{2302.11067}.\\

A {\bf (flat) $N$-clusteron} is a state of $N-1$ consecutive zeroes, i.e. $a_\bullet=1^{N-1}$ (where we use powers to express the concatenation of identical digits). This corresponds to $N$ violinists in consecutive rooms.\\

We will now describe a {\bf  move} from one state of the system  $a_\bullet^t$ to the next one $a_\bullet^{t+1}$.
\begin{itemize}
\item Pick $T$ with $a_T=0$ call $T$ the triggering position.
\item  Set $a_T=2$,
\item  find the biggest $j<T$ with $a_j>0$ and reduce it by one. If $j$ does not exist skip this step
\item  find the smallest $j>T$ with $a_j>0$ and reduce it by one. If $j$ does not exist skip this step
\end{itemize}

If no moves are possible (i.e. the sequence contains no more zeroes) we have reached a final state.  We will denote a sequence of moves by $a_\bullet^1,..., a_\bullet^n$ and the corresponding triggering positions by $T_1,...,T_{n-1}$. \\

I.e., a move is a 0-gap triggering by turning into a 2 and two minus ones propagating left resp. right until they meet a non-zero number (which they would reduce by one) or disappear at the boundary of the state.

Here are all possible sequences of moves for $N=4$ violinists, as described in Figure 1 of \cite{2302.11067}, in our notation:

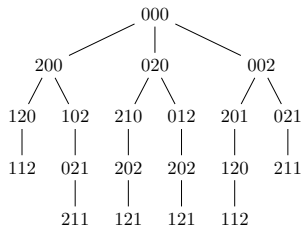
\begin{figure}[ht!]
    \centering
        \scalebox{0.7}{
            \begin{forest}
                [$000$
                [$200$
                [$120$[$112$]]
                [$102$[$021$[$211$]]]
                ]
                [$020$
                [$210$[$202$[$121$]]]
                [$012$[$202$[$121$]]]
                ]
                [$002$
                [$201$[$120$[$112$]]]
                [$021$[$211$]]
                ]
                ]
            \end{forest}
        }
        \caption{Possible move sequences for 4 violinists.}
    \label{fig:4exampletree}
\end{figure}

\section{Reaching a final state}
To show that we always reach a final state, we will use the fact that there is a non-decreasing quantity:
\begin{defi}
We call $|a_\bullet|=\sum_{i=1}^{N-1}a_i$ the norm of the state $a_\bullet$.
\end{defi}

We have the following property for norms of consecutive states $a_\bullet^t$, $a_\bullet^{t+1}$:
$$|a_\bullet^{t+1}|-|a_\bullet^t|\in\{0,1,2\},$$
where the value $2$ is only possible if $|a_\bullet^t|=0^{N-1}$. This means that the norm is non-decreasing and can only stay the same or increase by one (unless the initial states is a flat $N$-clusteron, then it will increase by 2).

\begin{thm} For any initial state $a_\bullet^1\in \mathbb N^{N-1}$, the final state is reached after finitely many steps.
\end{thm}

\begin{proof}Let $k=\max_{1\leq j\leq N-1}a_j$ and $K=\max(k,2)$. Since the components of $a$ either get smaller or get increased from 0 to 2,  the number of possible states starting from $a_\bullet^1$ is finite (a subset of $\{0,1,..., K\}^{n-1}$). Hence, either the assertion of the theorem is true or there is a sequence of $n>1$ states $a_\bullet^1,...,a_\bullet^n$ where the first one equals the last one. Let $i$ be the leftmost changing value, i.e. the smallest value such that $a_j^t=a_j^{t+1}$ for all $j<i$ and all $t$. Let now $t$ be the step in which $i$ changes, then
\begin{itemize}
    \item either there are only zeros left to $a_i^t$ , then $|a_\bullet^{t+1}|>|a_\bullet^t|$, which implies that the sequence can not be a cycle, because the norm is non-decreasing.
    \item or there is a non-zero number left of $i$, which then changes (reduces by 1) in step $t$ , which is a contradiction to our choice of $i$.
\end{itemize}
Hence there can be no cycle, i.e. we eventually run into a final state.
\end{proof}

\section{Characterization of  possible final states}
\cite{2302.11067} also gives a characterization of all possible final states when starting with a flat $N$-clusteron. Those are given by states containing just ones and one two (in an arbitrary position), i.e. they are of the form $1^k21^{N-2-k}$ for $k\in 0,...,N-2$. We will reprove this statement in this section.

First we observe that if we start with a state in the space $\{0,1,2\}^{N-1}$, then we will always stay in this space (gaps of length $>2$ will never appear). On the other hand, the norms of the states are bounded from above (even locally):

\begin{thm}
Starting with  $a_\bullet^1= 0^{N-1}$, there can never appear a state $a_\bullet^t$ containing a sub-sequence of length $l$ with norm higher than $l+1$.  I.e. there is no $l$ and $i$ such that $\sum_{i=0}^{l-1}a_{j+i}^t>l+1$.
\end{thm}

\begin{proof}
We do induction by length $l$. For $l=1$ this is just the statement that no numbers higher than $2$ can appear from a state in $\{0,1,2\}^{N-1}$. Assume we have proven the statement for lengths up to $l$. Further assume $t$ is the first step in which $\sum_{j=0}^{l}a_{i+j}^t>l+2$
We can further assume that the value of $\sum_{j=0}^{l}a_{i+j}^{t-1}$ is $l+2$. It is at most $l+2$ by assumption, and the only case where it would increase by more than one during step $t-1$  is if the whole sub-sequence $(a^{t-1}_i,...,a^{t-1}_{i+l})$ was zero, in which case $\sum_{j=0}^{l}a_{i+j}^t=2\leq l-2$.
For the sum  $\sum_{j=0}^{l}a_{i+j}^{t-1}$ to increase, either $a_i^{t-1}$ or $a_{i+l}^{t-1}$  has to be zero in which case
$\sum_{j=0}^{l-1}a_{i+1+j}^{t-1}$ or  $\sum_{j=0}^{l-1}a_{i+j}^{t-1}$ had value $l+2$ , contradicting the induction hypothesis.
\end{proof}

Final states never contain zeroes and any (non-initial) state contains at least one 2. Since by the above the norm of the state can never surpass $N-1+1=N$, we obtain:

\begin{cor}
Starting with  $a_\bullet^1= 0^{N-1}$, the final position always consists of one two and the rest ones, i.e. is of the form $1^k21^{N-2-k}$ for $k\in 0,...,N-2$.
\end{cor}

We still need to show that all positions of the two are possible, i.e. that $1^k21^{N-2-k}$ is a possible final state for all  $k\in 0,...,N-2$. We will prove this by induction and need the following statement:

\begin{lemma}Assume there is a sequence of moves to obtain a state $(a_0,...,a_{N-2})$ from $0^{N-2}$. Then there is a sequence of moves to obtain $(a_0,...,a_{i-1},0, a_i,...., a_{N-2})$ from $0^{N-1}$ for any $i\in 0,...,N-1$.
\end{lemma}
\begin{proof}
    If $T_1,...,T_t$ is the sequence of trigger points for $(a_0,...,a_{N-2})$, then we set
\[
\tilde T_k =
\begin{cases}
T_k & \text{if } T_k < i, \\
T_k + 1 & \text{if } T_k \geq  i.
\end{cases}
\]
and  $\tilde T_1,...,\tilde T_t$ become a sequence of trigger points leading to $(a_0,...,a_{i-1},0, a_i,...., a_{N-2})$ from $0^{N-1}$.
\end{proof}

We can now prove that

\begin{thm}
Starting with  $a_\bullet^1= 0^{N-1}$ all states of the form $1^k21^{N-2-k}$ for $k\in 0,...,N-2$ can be final.
\end{thm}
\begin{proof}
Again we proceed by induction. For $N-1=1$ (i.e. a single gap) there is only one final state $(2)$ possible. Assume we have proven the statement for $N-2$ gaps.
We now want to construct a sequence of moves turning $ 0^{N-1}$ into $1^k21^{N-2-k}$.

\begin{itemize}
    \item For $k=0$, we use the above Lemma and the induction hypothesis to arrive at the state $021^{N-3}$. Then we trigger the leftmost spot.
    \item  For $k=N-2$, we mirror the above construction, i.e. we generate $1^{N-3}20$ and trigger the rightmost position.
    \item For the other $k$ we use the induction hypothesis and Lemma to arrive at the state $01^{k-1}21^{N-2-k}$. By triggering the only possible zero repeatedly we will move the zero further and further to the right until we arrive at the desired state.

$$01....121...1\to 201...121....1\to   1201...121....1 \to   ... \to 1....12021...1 \to  1....121....1$$
\end{itemize}

\end{proof}

\printbibliography

\end{document}